\documentclass[11pt]{amsart}

\usepackage[colorlinks=false]{hyperref}
\usepackage{graphicx}
\usepackage{enumerate}

\theoremstyle{plain}
\newtheorem{theorem}{Theorem}[section]
\newtheorem{lemma}[theorem]{Lemma}
\newtheorem{proposition}[theorem]{Proposition}
\newtheorem{subproposition}{}[theorem]
\newtheorem{corollary}[theorem]{Corollary}

\theoremstyle{definition}
\newtheorem{definition}[theorem]{Definition}

\newcommand{\cl}{\operatorname{cl}}
\newcommand{\dash}{\nobreakdash-\hspace{0pt}}
\newcommand{\ba}{\backslash}

\title{Excluded minors for the class of split matroids}

\author[Cameron]{Amanda Cameron}
\address{School of Mathematical Sciences,
Queen Mary University of London,
London,
United Kingdom}
\email{a.cameron272@gmail.com}

\author[Mayhew]{Dillon Mayhew}
\address{School of Mathematics and Statistics,
Victoria University of Wellington,
New Zealand}
\email{dillon.mayhew@vuw.ac.nz}

\begin{document}

\maketitle
\begin{abstract}
Split matroids form a minor-closed class of matroids, and are defined by placing conditions on the system of split hyperplanes in the matroid base polytope.
They can equivalently be defined in terms of structural properties involving cyclic flats.
We confirm a conjecture of Joswig and Schr\"{o}ter by proving an excluded-minor characterisation of the class of split matroids.
\end{abstract}

\section{Introduction}

The class of split matroids was recently introduced by Joswig and Schr\"{o}ter \cite{joswig}, who successfully deployed them as a tool in tropical linear geometry.  The definition arises from natural considerations in the polyhedral view of matroids. Let $M$ be a matroid on the ground set $\{1,\ldots, n\}$. Any subset of $\{1,\ldots, n\}$ is identified with its characteristic vector in $\mathbb{R}^{n}$. The \emph{matroid base polytope}, $P(M)$, is the convex hull of the characteristic vectors of the bases of $M$. Roughly speaking, a split of a polytope is a division into two polytopes by a hyperplane, called a split hyperplane. If all pairs of split hyperplanes in a matroid polytope satisfy a certain compatibility condition, then the matroid is split. Although the motivation for split matroids arises from tropical linear geometry, natural questions also arise in the area of structural matroid theory, and it is one of these questions that we address here.

First we provide more detail on the polyhedral background.
Let $X$ be a set of points in $\mathbb{R}^{n}$.
The \emph{convex hull} of $X$ is the intersection of all closed half-spaces that contain $X$.
A \emph{polytope} is the convex hull of a finite set of points.
The intersection of two polytopes is also a polytope.
If $X$ is empty, then so is its convex hull.
Let $P$ be the convex hull of the non-empty finite set $X$.
Let $A$ be the affine subspace of $\mathbb{R}^{n}$ spanned by $P$, and let $H$ be any hyperplane of $A$.
Thus $A-H$ is partitioned into two open half-spaces of $A$.
If one of these has an empty intersection with $P$, then $H\cap P$ is a \emph{face} of $P$.
Note that the empty set is a face.
In addition, we declare $P$ itself to be a face of $P$.
A face is a \emph{facet} if it is properly contained in exactly one face, namely $P$.
A \emph{vertex} is a minimal non-empty face.
A point in $P$ that is in no face other than $P$ itself is in the \emph{relative interior} of $P$.
Every vertex of $P$ is a point in $X$.
Every face of $P$ is the convex hull of the vertices it contains, and is therefore a polytope.

The notion of a polytope split originated in \cite{BD92} (see \cite[Section 5.3.3]{LRS10}).
The definition we use here is from \cite{herrmann}. We let $P$ be a polytope. A \emph{split} of $P$ is a collection, $\mathcal{C}$, of polytopes  such that:
\begin{enumerate}[(i)]
    \item the empty polytope is in $\mathcal{C}$,
    \item if $Q$ is in $\mathcal{C}$, then all the vertices of $Q$ are also vertices of $P$,
    \item if $Q$ is in $\mathcal{C}$, so are all the faces of $Q$, 
    \item the intersection of any two distinct polytopes $Q_1,Q_2\in\mathcal{C}$ is a face of both $Q_1$ and $Q_2$,
    \item $\bigcup_{C\in \mathcal{C}} C=P$, and
    \item there are exactly two maximal polytopes in $\mathcal{C}$.
\end{enumerate}
The members of $\mathcal{C}$ are called the \emph{cells} of the split. The affine subspace spanned by the intersection of the two maximal cells is called a \emph{split hyperplane}.

Let $\Delta(r,n)$ be the $(n-1)$\dash dimensional \emph{hypersimplex}: that is, the convex hull of those $0,1$\dash vectors in $\mathbb{R}^{n}$ with exactly $r$ ones.
Hence $\Delta(r,n)$ is the base polytope of the uniform matroid $U_{r,n}$.
Note that the polytope of any rank-$r$ matroid on $n$ elements is contained in $\Delta(r,n)$.
Let $M$ be a rank-$r$ matroid with ground set $\{1,\ldots, n\}$.
If $x$ is in $\mathbb{R}^{n}$, then $x_{i}$ stands for the entry of $x$ indexed by $i\in E(M)$.
Edmonds \cite{edmonds} proved that
\[P(M)=\left\{x\in\Delta(r,n)\ \colon \sum_{i\in F}x_i\leq r(F) \ \text{for all flats}\ F\ \text{of}\ M\right\}.\]
Let $F$ be a flat of $M$.
Then $H(F)$ is the set
\[\left\{x\in\mathbb{R}^{n}\ \colon \sum_{i\in F}x_{i}=r(F)\right\}.\]
If $F$ is minimal under inclusion with respect to $H(F)$ intersecting $P(M)$ in a facet of $P(M)$, then we say that $F$ is a \emph{flacet} of $M$.
(This definition originated in \cite{FS05}.)
If, in addition, $H(F)\cap \Delta(r,n)$ spans a split hyperplane of $\Delta(r,n)$, then we say that $F$ is a \emph{split flacet} of $M$.
In this case, we can think of $H(F)$ as separating $P(M)$ from a portion of $\Delta(r,n)$ that does not intersect $P(M)$.
Roughly speaking, the split flacets are the hyperplanes we use when carving off portions of $\Delta(r,n)$ to obtain $P(M)$.

\begin{definition}[\cite{joswig}]
\label{def2}
Assume that $M$ is a rank-$r$ matroid with ground set $\{1,\ldots, n\}$. Let $A$ be the affine subspace of $\mathbb{R}^{n}$ spanned by $P(M)$. We use $[0,1]^{n}$ to denote the closed unit cube. Assume that the following holds for any distinct split flacets, $F_{1}$ and $F_{2}$, of $M$: no point in $H(F_{1})\cap H(F_{2})$ is in the relative interior of $A\cap [0,1]^{n}$. Then we say that $M$ is a \emph{split matroid}.
\end{definition}

Joswig and Schr\"{o}ter  observe that the matroid polytopes of split matroids are exactly those polytopes whose faces of codimension at least two are contained in the boundary of $\Delta(r,n)$. They use split matroids and the Dressian to construct a number of nonrepresentable tropical linear spaces, and give a characterisation of matroid representability in terms of these spaces. In addition, they prove that the class of split matroids contains the (possibly dominating) class of sparse paving matroids.

The following result is \cite[Proposition 44]{joswig}.

\begin{proposition}
\label{minorclosed}
The class of split matroids is closed under duality and under taking minors.
\end{proposition}

In light of Proposition \ref{minorclosed}, we naturally ask what the excluded minors are for the class of split matroids.
Joswig and Schr\"{o}ter identify five excluded minors.
Our main theorem shows that their list of excluded minors is complete.
Figure~\ref{exclminors} shows geometric representations of four connected rank-$3$ matroids, each with six elements.
Note that $S_{1}^{*}\cong S_{2}$, whereas $S_{3}$ and $S_{4}$ are both self-dual matroids.
In addition, we define $S_0$ to be the matroid constructed from the direct sum $U_{2,3}\oplus U_{2,3}$ by adding one parallel point to each of the two connected components.
Thus $S_{0}$ is the direct sum of two copies of $M(\mathcal{W}_{2})$, where $\mathcal{W}_{2}$ is the graph obtained by adding a parallel edge to a triangle.

\begin{figure}[htb]
\centering
\includegraphics{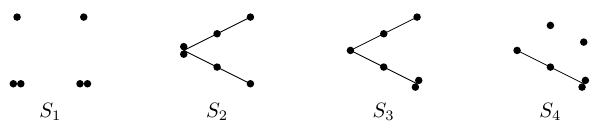}
\caption{Connected excluded minors for split matroids.}
\label{exclminors}
\end{figure}

\begin{theorem}
\label{maintheorem}
The excluded minors for the class of split matroids are $S_0$, $S_1$, $S_2$, $S_3$, and $S_4$.
\end{theorem}

Any unexplained matroid terms can be found in \cite{oxley}.

\section{Reducing to the connected case}

To prove Theorem \ref{maintheorem}, we employ Joswig and Schr\"{o}ter's equivalent formulation of Definition \ref{def2} that relies entirely on structural concepts.

We say that a flat, $Z$, of the matroid, $M$, is \emph{proper} if $0<r(Z)<r(M)$.
A set $X\subseteq E(M)$ is \emph{cyclic} if the restriction $M|X$ contains no coloop.
The next result is \cite[Proposition 1]{joswig}.

\begin{proposition}
\label{flacet}
Let $Z$ be a flat of the connected matroid $M$. Then $Z$ is a flacet if and only if it is proper, and both $M|Z$ and $M/Z$ are connected.
\end{proposition}

\begin{proposition}
\label{splitflacets}
Let $Z$ be a flat of the connected matroid $M$. Then $Z$ is a split flacet if and only if it is proper and cyclic, and both $M|Z$ and $M/Z$ are connected.
\end{proposition}

\begin{proof}
Let $E=\{1,\ldots, n\}$ be the ground set of $M$, and let $r$ be the rank of $M$.
Assume $Z$ is a proper cyclic flat of $M$ and that $M|Z$ and $M/Z$ are connected.
Then $Z$ is a flacet by Proposition \ref{flacet}.
We know that $0< r(Z) < |Z|$, since $Z$ is a proper flat and is not independent.
As $Z$ and $E-Z$ are non-empty, we can find an element in $E-Z$ that is not a coloop (since $M$ is connected).
These are the conditions required to apply Lemma 6 of \cite{joswig}.
From this lemma, we see that the equation
\[
(r-r(Z))\sum_{i\in Z}x_{i} = r(Z)\sum_{i\notin Z}x_{i},\ \text{or equivalently},\ 
r\sum_{i\in Z}x_{i} = r(Z)\sum_{i\in E}x_{i}.
\]
defines a split hyperplane of $\Delta(r,n)$.
The intersection of $H(Z)$ with $\Delta(r,n)$ satisfies
the equation $\sum_{i\in Z}x_{i}=r(Z)$.
By multiplying both sides of this equation by $r$, and using the fact that in $\Delta(r,n)$ we have the
equality $\sum_{i\in E}x_{i}=r$, we again obtain
\[
r\sum_{i\in Z}x_{i} = r(Z)\sum_{i\in E}x_{i}.
\]
This shows that the intersection $H(Z)\cap \Delta(r,n)$ is a split of $\Delta(r,n)$, so $Z$ is a split
flacet, as desired.

For the converse, we let $Z$ be a split flacet.
Then $Z$ is a proper flat and both $M|Z$ and $M/Z$ are connected by Proposition \ref{flacet}.
We need only show $Z$ is cyclic.
Since $H(Z)\cap \Delta(r,n)$ is a split of $\Delta(r,n)$, \cite[Proposition 4]{joswig} asserts there is a positive integer, $\mu$, which satisfies $r > \mu > r-|Z|$.
This implies $|Z| > 1$.
Proposition 13 in \cite{joswig} says that any flacet of $M$ with at least two elements is a cyclic flat.
Therefore $Z$ is cyclic and the proof is complete.
\end{proof}

\begin{definition}
\label{connected}
Let $M$ be a connected matroid, and let $Z$ be a proper cyclic flat of $M$. If both $M|Z$ and $M/Z$ are connected matroids, but at least one of them is a non-uniform matroid, we say that $Z$ is a \emph{certificate for non-splitting}.
\end{definition}

\begin{lemma}
\label{certificates}
Let $M$ be a connected matroid.
Then $M$ is split if and only if it has no certificate for non-splitting.
\end{lemma}

\begin{proof}
Theorem 11 in \cite{joswig} states that $M$ is split if and only if $M|Z$ and $M/Z$ are both uniform, for every split flacet $Z$.
So the lemma follows immediately from Proposition \ref{splitflacets}.
\end{proof}

The following result combines Lemma 10 and Proposition 15 of \cite{joswig}.

\begin{proposition}
\label{disconnected}
Let $U_{1},\ldots, U_{t}$ be the connected components of the matroid $M$, where $t>1$. Then $M$ is a split matroid if and only if each connected matroid, $M|U_{i}$, is a split matroid, and at most one of these matroids is non-uniform.
\end{proposition}

Note that a characterisation of connected split matroids now immediately leads to a characterisation of all split matroids, by use of Proposition \ref{disconnected}.
Note also that $S_{0}$ has two connected components that are non-uniform, and hence $S_{0}$ is not a split matroid.
It is also easy to check that $S_{0}$ is an excluded minor for the class of split matroids. We now show that it is the only disconnected excluded minor. The following result is a consequence of \cite[Theorem 4.1]{gershkoffoxley}.

\begin{proposition}
\label{wheelminor}
Every connected non-uniform matroid has an $M(\mathcal{W}_{2})$\dash minor.
\end{proposition}

\begin{proposition}
\label{disconnectedcase}
The only disconnected excluded minor for the class of split matroids is $S_{0}$.
\end{proposition}

\begin{proof}
Suppose $M$ is a disconnected excluded minor, so $M$ is not a split matroid, but every proper minor of $M$ is.
Let the connected components of $M$ be $U_{1},\ldots, U_{t}$, where $t>1$.
As each $M|U_{i}$ is a proper minor of $M$, we see that $M|U_{i}$ is a split matroid for each $i$.
If at most one component of $M$ is non-uniform, then $M$ is split, which is a contradiction.
So let $M|U_i$ and $M|U_j$ be non-uniform, where $1\leq i < j \leq t$.
Now $M|(U_{i}\cup U_{j})$ has two components, $U_{i}$ and $U_{j}$.
Both $M|U_{i}$ and $M|U_{j}$ are split but non-uniform, so $M|(U_{i}\cup U_{j})$ is not split.
Therefore it cannot be a proper minor of $M$.
From this we deduce that $i=1$ and $j=t=2$.
By Proposition \ref{wheelminor}, each of the two components of $M$ contains $M(\mathcal{W}_{2})$ as a minor.
Hence $M$ contains a minor isomorphic to $S_{0}\cong M(\mathcal{W}_{2})\oplus M(\mathcal{W}_{2})$.
As $S_{0}$ is an excluded minor, and no excluded minor can properly contain another, we now see that $M$ is isomorphic to $S_{0}$, as desired.
\end{proof}

\section{Proof of the main theorem}

\begin{lemma}
\label{S2S3S4}
Let $M$ be a connected matroid.
If $M$ has a proper cyclic flat, $Z$, such that $M|Z$ is connected and has an $M(\mathcal{W}_{2})$\dash minor, then $M$ has a minor isomorphic to $S_{2}$, $S_{3}$, or $S_{4}$.
\end{lemma}

\begin{proof}
Let $M$ be a counterexample chosen so that its ground set is as small as possible.
We let $Z$ be a proper cyclic flat of $M$ such that $M|Z$ is connected with an $M(\mathcal{W}_{2})$\dash minor.
Amongst all such flats, we assume that we have chosen $Z$ to be as small as possible.
Since $M$ is a counterexample, it has no minor isomorphic to $S_{2}$, $S_{3}$, or $S_{4}$.

\begin{subproposition}
\label{subprop2}
If $e$ is any element of $Z$, then $(M|Z)\ba e$ has no $M(\mathcal{W}_{2})$\dash minor.
\end{subproposition}

\begin{proof}
We assume otherwise.
It is well-known and easy to verify that $Z-e$ is a flat of $M\ba e$.
As $Z$ contains an $M(\mathcal{W}_{2})$\dash minor, we see that $|Z| \geq 4$.
First we consider the case that $(M|Z)\ba e = M|(Z-e)$ is connected.
Since $M|(Z-e)$ is a connected, non-empty matroid, it contains no coloops.
This shows that $Z-e$ is a cyclic flat of $M\ba e$.
Since $M|(Z-e)$ has an $M(\mathcal{W}_{2})$\dash minor, it has rank greater than zero.
As $e$ is not a coloop of $M$, or of $M|Z$, we also have $r_{M\ba e}(Z-e)=r_{M}(Z)<r(M)=r(M\ba e)$.
This establishes that $Z-e$ is a proper cyclic flat of $M\ba e$.
Assume that $M\ba e$ is not connected, and let $(U,V)$ be a separation.
Since $M|(Z-e)$ is connected, we can assume that $Z-e$ is a subset of $U$.
As $Z$ is a cyclic flat, $e$ is spanned by $Z-e$ in $M$.
From this it follows that $(U\cup e,V)$ is a separation of $M$, which is impossible.
Therefore $M\ba e$ is a connected matroid, and $Z-e$ is a proper cyclic flat of $M\ba e$ such that $(M\ba e)|(Z-e) = (M|Z)\ba e$ is connected and has an $M(\mathcal{W}_{2})$\dash minor.
We have shown that $M\ba e$ is a smaller counterexample to the lemma, and from this contradiction
we deduce that $(M|Z)\ba e$ is not connected.

Let $(U_{1},\ldots, U_{t})$ be the partition of $Z-e$ into connected components of $(M|Z)\ba e$, where $t>1$.
Thus $(M|Z)\ba e = (M|U_{1})\oplus\cdots\oplus (M|U_{t})$.
Since $M(\mathcal{W}_{2})$ is a connected matroid, we can assume that $M|U_{1}$ has an $M(\mathcal{W}_{2})$\dash minor \cite[Proposition 4.2.20]{oxley}.
As $U_{1}$ is a connected component of $(M|Z)\ba e$ with at least four elements there are no coloops in $M|U_{1}$.
It follows that $U_{1}$ is a cyclic flat of $(M|Z)\ba e$.
Assume that $U_{1}$ is not a flat of $M$, and let $z$ be an element in $\cl_{M}(U_{1})-U_{1}$.
Note that $\cl_{M}(U_{1})\subseteq \cl_{M}(Z)=Z$, so $z$ is in $Z$.
If $z=e$, then $(U_{1}\cup e,U_{2}\cup\cdots\cup U_{t})$ is a separation of the connected matroid $M|Z$, so $z\ne e$.
Let $C$ be a circuit containing $z$ such that $C\subseteq U_{1}\cup z$.
Then $C$ contains elements from both $U_{1}$ and $U_{2}\cup\cdots \cup U_{t}$, and as $(U_{1},U_{2}\cup\cdots\cup U_{t})$ is a separation of $(M|Z)\ba e$, we have a contradiction.
Therefore $U_{1}$ is a cyclic flat of $M$.
Now $r_{M}(U_{1})\leq r_{M}(Z)<r(M)$, and obviously $r_{M}(U_{1})>0$, so $U_{1}$ is a proper cyclic flat of $M$.
Moreover $M|U_{1}$ is connected and has an $M(\mathcal{W}_{2})$\dash minor.
But we chose $Z$ to be the smallest possible cyclic flat with these properties, and $U_{1}$ does not contain any element of $U_{2}\cup\cdots\cup U_{t}$ so it is strictly smaller than $Z$.
This contradiction completes the proof.
\end{proof}

\begin{subproposition}
\label{subprop1}
If $x$ is an element in the complement of $Z$, then $M\ba x$ is not connected.
\end{subproposition}

\begin{proof}
Assume otherwise.
Note that $r_{M\ba x}(Z)=r_{M}(Z)<r(M)=r(M\ba x)$, so it is obvious that $Z$ is a proper cyclic flat of $M\ba x$.
Moreover $(M\ba x)|Z = M|Z$ is connected and has an $M(\mathcal{W}_{2})$\dash minor.
This contradicts the minimality of $M$, so $M\ba x$ is not connected.
\end{proof}

\begin{subproposition}
The complement of $Z$ is a series pair of $M$.
\end{subproposition}

\begin{proof}
Choose an arbitrary element, $x$, in the complement of $Z$.
Using \ref{subprop1}, we let $(U_{1},\ldots, U_{t})$ be the partition of $E(M)-x$ into connected components of $M\ba x$, where $t>1$.
As $M|Z$ is connected, we can assume that $Z\subseteq U_{1}$.
Then $Z$ is a cyclic flat of $M|U_{1}$.
If it is a proper cyclic flat of $M|U_{1}$, then $M|U_{1}$ is a connected matroid with a proper cyclic flat such that the restriction to this cyclic flat is connected with  an $M(\mathcal{W}_{2})$\dash minor.
This contradicts the minimality of $M$, so $Z$ spans $U_{1}$.
It is straightforward to verify that $U_{1}$ is a flat of $M$, using some of the same arguments as in \ref{subprop2}.
Hence $Z=U_{1}$.

Let $y$ be an element of $U_{2}$.
Again using \ref{subprop1}, we see that $M\ba y$ is not connected.
Therefore $M/y$ is connected \cite[Theorem 4.3.1]{oxley}.
We can easily check that $\cl_{M/y}(Z)$ is a cyclic flat of $M/y$, and that $(M/y)|(\cl_{M/y}(Z))$ is connected with an $M(\mathcal{W}_{2})$\dash minor.
So if $\cl_{M/y}(Z)$ is a proper cyclic flat of $M/y$, we have contradicted the minimality of $M$.
Therefore $Z$ is not a proper cyclic flat of $M/y$, meaning that $r(Z)=r(M)-1$.
Hence $Z$ is a hyperplane of $M$, and its complement is a cocircuit.
However,
\[r(M)=r(M\ba x)=r(U_{1})+\cdots +r(U_{t}) = r(Z)+r(U_{2})+\cdots + r(U_{t}).\]
From this, and the fact that $M$ has no loops, we deduce that $t=2$, and that $r(U_{2})=1$.
Assume that $|U_{2}|>1$, and let $z$ be an element in $U_{2}-y$.
Then $\{y,z\}$ is a parallel pair.
But deleting an element from a parallel pair in a connected matroid always produces another connected matroid, so we are led to a violation of \ref{subprop1}.
Thus $U_{2}=\{y\}$, and we conclude that the complement of $Z$ is the series pair $\{x,y\}$.
\end{proof}

Let $\{x,y\}$ be the complement of $Z$, so that $\{x,y\}$ is a series pair.
Since $M|Z$ has an $M(\mathcal{W}_{2})$\dash minor, but \ref{subprop2} implies we cannot produce such a minor by deleting any element, we see that there is a subset $I\subseteq Z$ such that $(M|Z)/I$ is isomorphic to $M(\mathcal{W}_{2})$.
Assume $I$ is not independent, and let $e$ be an element contained in a circuit of $M|I$.
Then $(M|Z)/I=(M|Z)/(I-e)\ba e$, so we have a contradiction to \ref{subprop2}.
Therefore $I$ is an independent set.
Dualising, we see that $(M|Z)^{*} = (M\ba\{x,y\})^{*}=M^{*}/\{x,y\}$ has a coindependent set, $I$, such that $M^{*}/\{x,y\}\ba I$ is isomorphic to $M(\mathcal{W}_{2})$ (as $M(\mathcal{W}_{2})$ is self-dual).
Note that $\{x,y\}$ is a parallel pair in $M^{*}$.
As $I$ is coindependent, $r(M^{*}/\{x,y\})=r(M^{*}/\{x,y\}\backslash I)=r(M(\mathcal{W}_{2}))=2$, so $r(M^{*})=3$.

We choose elements $a$, $b$, $c$, and $d$, so that $(M^{*}/\{x,y\})|\{a,b,c,d\}$ is isomorphic to $M(\mathcal{W}_{2})$, where $\{a,b\}$ is a parallel pair in $M^{*}/\{x,y\}$.
Note that $\{a,b,x\}$ has rank two in $M^{*}$, that $\{c,d,x\}$ is independent, and that neither $c$ nor $d$ is on the line spanned by $\{a,b,x\}$.
We divide into two cases, according to whether or not $\{a,b\}$ is a parallel pair in $M^{*}$.

First assume that $\{a,b\}$ is not a parallel pair, so that it is independent in $M^{*}$.
Note that $M^{*}|\{a,b,x,y\}$ is isomorphic to $M(\mathcal{W}_{2})$.
The lines $\cl_{M}^{*}(\{c,d\})$ and $\cl_{M}^{*}(\{a,b,x,y\})$ intersect in a flat of rank at most one, and this flat cannot contain $x$.
Hence the intersection of $\cl_{M}^{*}(\{c,d\})$ and $\{a,b,x,y\}$ is either empty, or it contains $a$ (up to symmetry between $a$ and $b$).
In the first case, the restriction $M^{*}|\{a,b,c,d,x,y\}$ is isomorphic to $S_{4}$,
and in the second it is isomorphic to $S_{3}$.
In these cases, $M$ also has a minor isomorphic to $S_{3}$ or $S_{4}$.
Since this is a contradiction, we assume that $\{a,b\}$ is a parallel pair of $M^{*}$.

If $\{a,c,d\}$ is independent, then $M^{*}|\{a,b,c,d,x,y\}$ is isomorphic to $S_{1}$, which implies that $M$ has a minor isomorphic to $S_{1}^{*}\cong S_{2}$.
This is a contradiction, so $\{a,c,d\}$ has rank two.
Note that the restriction to $\{a,b,c,d\}$ is isomorphic to $M(\mathcal{W}_{2})$.
As $M^{*}$ is a connected rank\dash $3$ matroid, the complement of the line $\cl_{M}^{*}(\{a,b,c,d\})$ has rank at least two.
We let $z$ be an element in this complement, chosen so that $\{x,z\}$ is independent.
The intersection of $\cl_{M}^{*}(\{x,y,z\})$ and $\{a,b,c,d\}$ is either $\emptyset$, $\{a,b\}$, or $\{c\}$ (up to symmetry between $c$ and $d$).
In the first case, $M^{*}|\{a,b,c,d,x,z\}$ is isomorphic to $S_{4}$.
In the second and third cases, $M^{*}|\{a,c,d,x,y,z\}$ is isomorphic to $S_{3}$.
Thus we have a contradiction in any case, and this completes the proof of the lemma.
\end{proof}

\begin{proposition}
\label{pcf}
Let $Z$ be a proper cyclic flat of the matroid $M$. If $E(M)-Z$ is not a proper cyclic flat of $M^*$, then every element in $E(M)-Z$ is a coloop of $M$.
\end{proposition}

\begin{proof}
Let $E$ be the ground set of $M$.
The fact that $E-Z$ is a cyclic flat of $M^*$ is well-known and easy to verify.
Suppose it is not proper; that is, $r^*(E-Z)=r(M^*)$ or $r^*(E-Z)=0$.
First, consider the case where $r^*(E-Z)=r(M^*)=|E|-r(M)$.
Then the corank function gives \[|E|-r(M)=r(Z)+|E-Z|-r(M).\]
This implies that $r(Z)=|Z|$, so $Z$ is an independent set in $M$.
The only independent cyclic flat is the empty set, and $Z$ is non-empty since it is a proper flat of $M$.
So if $E-Z$ is not a proper flat, then $r^{*}(E-Z)=0$, and this implies that every element in $E-Z$ is a coloop of $M$.
\end{proof}

\begin{proposition}
\label{dualitycertificate}
Let $M$ be a connected matroid that is not split.
There exists $M'\in\{M,M^{*}\}$ such that the following holds: $M'$ has a proper cyclic flat, $Z$, where $M'|Z$ is connected and non-uniform.
\end{proposition}

\begin{proof}
Let $E$ be the ground set of $M$.
As $M$ is connected and not split, it contains a certificate, $Z$, for non-splitting, by Lemma \ref{certificates}.
Thus $Z$ is a proper cyclic flat such that both $M|Z$ and $M/Z$ are connected matroids and either $M|Z$ or $M/Z$ is non-uniform.
If $M|Z$ is non-uniform, then we set $M'$ to be $M$ and we are done.
So we assume that $M/Z$ is non-uniform.
If $M$ contains a coloop, then it is isomorphic to the uniform matroid $U_{1,1}$, and is therefore a split matroid.
This is impossible, so $M$ has no coloops.
We apply Proposition \ref{pcf} and deduce that $E-Z$ is a proper cyclic flat of $M^{*}$.
Note that $M^{*}|(E-Z)=(M/Z)^{*}$ and $M^{*}/(E-Z)=(M|Z)^{*}$.
Both of these matroids are connected, and $M^{*}|(E-Z)=(M/Z)^{*}$ is non-uniform.
Therefore we set $M'$ to be $M^{*}$ and relabel $E-Z$ as $Z$.
\end{proof}

As an aside, this gives us an alternative characterisation of connected split matroids.

\begin{corollary}
\label{maincorollary}
Let $M$ be a connected matroid.
Then $M$ is split if and only if, for every $M'\in\{M,M^{*}\}$ and every proper cyclic flat $Z$ of $M'$, when $M'|Z$ is connected it is uniform.
\end{corollary}

\begin{proof}
Proposition \ref{dualitycertificate} provides us with the ``if" direction.
For the ``only if" direction, we assume there exists $M' \in \{M,M^{*}\}$ and a proper cyclic flat $Z$ of $M'$ such that $M'|Z$ is connected and non-uniform.
By Proposition \ref{wheelminor}, $M'|Z$ has an $M(\mathcal{W}_{2})$\dash minor.
Lemma \ref{S2S3S4} tells us that $M'$ has a minor isomorphic to $S_{2}$, $S_{3}$, or $S_{4}$, so $M$ has a minor isomorphic to $S_{1}$, $S_{2}$, $S_{3}$, or $S_{4}$.
This implies $M$ is not split.
\end{proof}

We can now easily prove our main result.

\begin{proof}[Proof of \textup{Theorem \ref{maintheorem}}.]
The connected matroids $S_{1}$, $S_{2}$, $S_{3}$, and $S_{4}$ all contain certificates for non-splitting.
It is routine to verify that they are excluded minors.
Let $M$ be an excluded minor for the class of split matroids.
If $M$ is not connected, then it is isomorphic to $S_{0}$ by Proposition \ref{disconnectedcase}.
Therefore we assume that $M$ is connected.
By using Proposition \ref{dualitycertificate} and duality, we can assume that $M$ has a proper cyclic flat, $Z$, such that $M|Z$ is connected and non-uniform.
Proposition  \ref{wheelminor} implies that $M|Z$ has an $M(\mathcal{W}_{2})$\dash minor.
Lemma \ref{S2S3S4}, and the fact that no excluded minor properly contains another, implies that $M$ is isomorphic to $S_{2}$, $S_{3}$, or $S_{4}$.
(Note that $S_{1}$ does not appear in this analysis because of our duality assumption.)
\end{proof}

\section{Acknowledgments}

We thank James Oxley and Michael Joswig for their helpful advice, and the referees for useful comments. Dillon Mayhew was supported by a Rutherford Discovery Fellowship.


\begin{thebibliography}{1}

\bibitem{BD92}
H.-J. Bandelt and A.~W.~M. Dress.
\newblock A canonical decomposition theory for metrics on a finite set.
\newblock \emph{Adv. Math.} \textbf{92} (1992), no.~1, 47--105.

\bibitem{LRS10}
J.~A. De~Loera, J.~Rambau, and F.~Santos.
\newblock \emph{Triangulations}, volume~25 of \emph{Algorithms and Computation
  in Mathematics}.
\newblock Springer-Verlag, Berlin (2010).
\newblock Structures for algorithms and applications.

\bibitem{dress}
A.~W.~M. Dress and W.~Wenzel.
\newblock Valuated matroids.
\newblock \emph{Advances in Mathematics} \textbf{93} (1992), no.~2, 214--250.

\bibitem{edmonds}
J.~Edmonds.
\newblock Submodular functions, matroids, and certain polyhedra.
\newblock In \emph{Combinatorial Optimization - Eureka, You Shrink!, Papers
  Dedicated to Jack Edmonds, 5th International Workshop, Aussois, France, March
  5-9, 2001, Revised Papers} (2001) pp. 11--26.

\bibitem{FS05}
E.~M. Feichtner and B.~Sturmfels.
\newblock Matroid polytopes, nested sets and {B}ergman fans.
\newblock \emph{Port. Math. (N.S.)} \textbf{62} (2005), no.~4, 437--468.

\bibitem{gershkoffoxley}
Z.~Gershkoff and J.~Oxley.
\newblock A notion of minor-based matroid connectivity.
\newblock \emph{Adv. in Appl. Math.} \textbf{100} (2018), 163--178.

\bibitem{herrmann}
S.~Herrmann and M.~Joswig.
\newblock Splitting polytopes.
\newblock \emph{M\"unster Journal of Mathematics} \textbf{1} (2008), 109--141.

\bibitem{joswig}
M.~Joswig and B.~Schr\"oter.
\newblock Matroids from hypersimplex splits.
\newblock \emph{Journal of Combinatorial Theory Series A} \textbf{151} (2017),
  254--284.

\bibitem{oxley}
J.~Oxley.
\newblock \emph{Matroid theory}.
\newblock Oxford Graduate Texts in Mathematics. Oxford University Press,
  Oxford, second edition (2011).

\end{thebibliography}

\end{document}